\documentclass{amsart}
\usepackage{amsmath}
\usepackage{amsmath}
\usepackage{amssymb}

\title[$C^*$-envelopes of universal free products]{$C^*$-envelopes
of universal free products and semicrossed products for
multivariable dynamics}

\author{Benton L. Duncan}

\address{Department of Mathematics\\
300 Minard Hall\\
North Dakota State University\\
Fargo, ND  58105-5075\\
USA}

\email{benton.duncan@ndsu.edu}



\begin{document}

\theoremstyle{plain}
\newtheorem{thm}{Theorem}
\newtheorem{lem}{Lemma}
\newtheorem{prop}{Proposition}
\newtheorem{cor}{Corollary}

\theoremstyle{definition}
\newtheorem{dfn}{Definition}
\newtheorem*{construction}{Construction}
\newtheorem*{example}{Example}

\theoremstyle{remark}
\newtheorem*{conjecture}{Conjecture}
\newtheorem*{acknowledgement}{Acknowledgements}
\newtheorem{rmk}{Remark}

\begin{abstract} We show that for a class of operator algebras
satisfying a natural condition the $C^*$-envelope of the universal
free product of operator algebras $A_i$ is given by the free product
of the $C^*$-envelopes of the $A_i$. We apply this theorem to, in
special cases, the $C^*$-envelope of the semicrossed products for
multivariable dynamics in terms of the single variable semicrossed
products of Peters.
\end{abstract}

\maketitle

An important invariant for non-selfdajoint operator algebras is the
$C^*$-envelope.  This is a minimal $C^*$-algebra containing the
operator algebra in a completely isometric manner.  The utility of
such a $C^*$-algebra was laid out in \cite{Arveson:1969} and
\cite{Arveson:1972} and its existence was proved by Hamana in
\cite{Hamana:1979} using injective envelopes.

Unfortunately as with most universal objects identifying the
requisite $C^*$-algebra is often difficult, and is often carried out
on a case by case basis.  There have been several important classes
of operator algebras which have received intensive study: the
semigroupoid algebras of \cite{Kribs-Power:2003a} as special cases
of the tensor algebras over $C^*$-correspondences of
\cite{Muhly-Solel:1998}, and the semicrossed products of
\cite{Peters:1984}.  Both algebras try to encode some sort of
dynamics on an underlying $C^*$-algebra.  In both of these cases
however the dynamics are constrained significantly by either
avoiding interactions between morphisms as in the first algebras, or
by constraining the dynamics to a single variable in the case of the
second algebra.

Recently a new attempt at multivariable dynamics has been initiated
in \cite{Davidson-Katsoulis:2007}.  There, two possible universal
objects related to a multivariable system of dynamics are defined
and studied.  In particular they let $\tau = ( \tau_1, \tau_2,
\cdots, \tau_n)$ be a tuple of continuous self maps of a locally
compact Hausdorff space.  They then study universal operator
algebras which encode these dynamics.  To do this they look at the
universal operator algebra generated by $C_0(X)$ and contractions
$S_i$ encoding the dynamics of $ \tau_i$ via a covariance relation
$S_if(x) = f(\tau_i(x))S_i$ for all $ f \in C_0(X)$.  There are two
universal operator algebras they study the first they call the
semicrossed product, and the second the tensor algebra.  The only
difference being an additional constraint on the tensor algebra,
that the contractions $S_i$ are a family of row contractions.  This
additional constraint allows a cleaner analysis and more concrete
theorems.  In particular, building on the groundbreaking work in
\cite{Katsoulis-Kribs:2006} and \cite{Muhly-Solel:1998}, the
$C^*$-envelopes of the tensor algebras are identified in Theorem 5.1
of \cite{Davidson-Katsoulis:2007}.

The semicrossed products however are less tractable since they lack
this constraint.  While some results can be proved in these examples
they are often less satisfying.  In particular a good understanding
of the $C^*$-envelopes is lacking.  In this paper we begin to
address this issue by recognizing the semicrossed products for
multivariable dynamics as universal free products of the semicrossed
products of \cite{Peters:1984}.  We then apply a result for
$C^*$-envelopes of free products in the context of certain
semicrossed products.

In the first section we remind the reader of the universal objects
we study and their universal properties.  First are the universal
free products which were constructed intrinsically in
\cite{Blecher-Paulsen:1991}.  We then focus on the $C^*$-envelope a
now standard object in the non-selfadjoint operator algebra
literature.  We refer the reader to the books \cite{Paulsen:2002}
and \cite{Blecher-Lemerdy:2004} for very readable accounts on both
of these objects.

In the second section we prove some easy facts about the unique
extension property for completely positive maps following
\cite{Arveson:2005} and \cite{Dritschel-Mccullouch:2006}. We then
prove in the third section Theorem \ref{main} that allows us to
relate the $C^*$-envelope and universal free products. In the final
two sections we bring these results together and apply them to
certain semicrossed products for multivariable dynamics to calculate
the $C^*$-envelope of the semicrossed products.

\section{Free products and $C^*$-envelopes of operator algebras}

We begin by reminding the reader of the universal properties for
free products and $C^*$-envelopes.

We assume the following, $\{ A_i \}$  is a collection of operator
algebras, sharing a common $C^*$-subalgebra which we will call $D$.
We will denote by $*_D A_i$  the universal operator algebra free
product with amalgamation over $D$.  In particular, we mean the
universal operator algebra satisfying the following universal
property.

\begin{center}{\em Universal Property for $*_D A_i$:}\end{center}
The algebra $*_DA_i$ is the universal operator algebra such that:
\begin{enumerate} \item there exists completely isometric
isomorphisms $\iota_i:A_i \rightarrow
*_DA_i$
\item given $\pi_i:A_i \rightarrow S$ completely contractive homomorphisms into
the operator algebra $S$, such that for all $i$ and $j$ we have $
\pi_i(x) = \pi_j(x)$ for all $x \in D$ there is a unique completely
contractive homomorphism $* \pi_i: *_D A_i \rightarrow S$ such that
$ (* \pi_i) \circ \iota_i = \pi_i$. \end{enumerate}

Notice that when the $A_i$ are $C^*$-algebras then $*_DA_i$ is a
$C^*$-algebra.

Given an operator algebra $A$ there is a unique $C^*$-algebra,
denoted $C^*_e(A)$ satisfying the following universal property.

\begin{center}{\em Universal Property for $C^*_e(A)$}: \end{center}
The $C^*$-algebra $C^*_e(A)$ is the universal $C^*$-algebra such
that: \begin{enumerate} \item there is a completely isometric
isomorphism $\iota_A: A \rightarrow C^*_e(A)$
\item the set $\iota_A(A)$ generates $C^*_e(A)$ as a $C^*$-algebra and
\item given a completely isometrically isomorphism $\pi: A \rightarrow C$
where $C$ is a $C^*$-algebra generated by $ \pi(A)$, there is a
unique onto $*$-homomorphism $\tilde{\pi}: C \rightarrow C^*_e(A)$
such that $ \tilde{\pi} \circ \pi = \iota_A$.\end{enumerate}

\section{The unique extension property and the main theorem}

We refer the reader to \cite{Arveson:2005} for a discussion of
boundary representations, Silov ideals, and the unique extension
property of unital completely positive maps.  Here we remind the
reader of some definitions and then prove a simple proposition that
is helpful in the context of universal free products.

If $A$ is a unital operator algebra and $C$ is a $C^*$-algebra such
that there is a completely isometric representation $\iota: A
\rightarrow C$, with $C^*(\iota(A)) = C$ we say that $C$ is
generated by $A$.  In this case we will usually drop reference to
the map $\iota$ and just denote the $C^*$-algebra by $C^*(A)$. Given
such a $C^*$-algebra and using the universal property for $C^*_e(A)$
there is a $*$-representation $\sigma: C^*(A) \rightarrow C^*_e(A)$.
The kernel of this representation is called the {\em Silov ideal}.

Now we say a unital $*$-representation $\pi: C^*(A) \rightarrow
B(\mathcal{H})$ is a {\em unique extension} if given $\tau: C^*(A)
\rightarrow B(\mathcal{H})$ a unital completely positive map such
that $ \tau|_A = \pi|_A$, then $ \tau = \pi$. It is shown in
\cite{Arveson:2005} that the Silov ideal will be contained in the
kernel of any representation which is a unique extension.  In the
terminology of Arveson we are defining $\pi$ to be a unique
extension if the map $\pi|_A$ has the unique extension property.

We will need a result relating the $C^*$-envelope, and the unique
extension property for unital completely positive maps.

\begin{prop}\label{boundary} If $A$ is a unital operator algebra and
$\pi: C^*(A) \rightarrow B(\mathcal{H})$ is a faithful
representation such that $\pi$ is a unique extension, then
$C^*(A)\cong C^*_e(A)$. \end{prop}

\begin{proof} Since $\pi$ is a unique extension we
know \cite[Proposition 2.2]{Arveson:2005} that $\pi$ is a maximal
unital completely positive map.  Hence, from the proof of
\cite[Corollary 3.3]{Arveson:2005} we see that the Silov ideal for
$A$ is contained in $ \ker \pi$.  But $ \ker \pi$ is trivial since
$\pi$ is faithful and hence $ C^*(A) \cong C^*_e(A)$.\end{proof}

\section{The main theorem}

In this section we present our main theorem relating the
$C^*$-envelope of a universal free product to the free product of
the $C^*$-envelopes.  First let us say that $A$ has the {\em unique
extension property} if every faithful $*$-representation of
$C^*_e(A)$ is a unique extension.

\begin{thm}\label{main} Let $A_i$ be a collection of unital operator
algebras, with common unital $C^*$-subalgebra $D$. If $A_i$ has the
unique extension property and $ C^*(A_i) \cap C^*(A_j) = D$ for all
$i \neq j$ when viewed as subalgebras of $C^*_e( *_D A_i)$ then
$C^*_e(*_DA_i)$ is $*$-isomorphic to $*_D C^*_e(A_i)$. \end{thm}

\begin{proof} We let $\pi: *_D C^*_e(A_i) \rightarrow
B(\mathcal{H})$ be a faithful $*$-representation of $*_DC^*_e(A_i)$.
We wish to show that $ \pi$ is a unique extension. We then apply
Proposition \ref{boundary} to get the result.

To do this let $ \tau: *_D C^*_e(A_i) \rightarrow B(\mathcal{H}) $
be a unital completely positive map such that $ \tau|_{*_D A_i} =
\pi|_{*_D A_i}$.  We need to show that $ \tau = \pi$.  First notice
that $ \pi_i := \pi|_{C^*_e(A_i)}$ is a faithful representation of
$C^*_e(A_i)$ and hence $\pi_i$ is a unique extension relative to
$A_i$.  Next we see that $ \tau|_{A_i} = \pi_i$, and hence $
\tau_i:= \tau|_{C^*_e(A_i)} = \pi_i$.  Now $\tau$ is a unital
completely positive map, hence when we apply Theorem 3.18 of
\cite{Paulsen:2002} we see that $\tau$ is a $*$-representation on
the $C^*$-algebra generated by $ \{ C^*_e(A_i) \}$.  It follows that
$\tau$ is a $*$-representation and hence $\tau = \pi$. \end{proof}

\begin{rmk} If the operator algebras $A_i$ have contractive approximate
identities, then we can as in \cite{Meyer:2001} adjoin a unit to
$A_i$ in a unique way such that $A_i$ imbeds completely
isometrically into the unitization $(A_i)^+$.  The free product $*_D
A_i$ will then be the subalgebra of $*_{D^+} (A_i^+)$ generated by
the $A_i$.  In a similar manner we have that $ C^*_e( *_D(A_i))$
will be the $C^*$-subalgebra of $*_{D^+}C^*_e(A^+)$ generated by the
$A_i$.  This $C^*$-algebra, however is completely isometrically
isomorphic to $*_D A_i$ and hence the result still carries through
for algebras with contractive approximate identities. \end{rmk}

This theorem can be seen as an analogue of a similar result for the
maximal $C^*$-dilation of the free products, see
\cite{Blecher:1999}. The proof is more complicated however, since
the $C^*$-envelope has the opposite universal property that one
would want.

Of course to apply this theorem we need to know when an operator
algebra has the unique extension property.  In the paper
\cite{Blecher-Labuschagne:2003} a seemingly stronger property was
shown to hold for the hierarchy of classes of operator algebras:
\[\begin{array}{c} \{\mbox{Operator algebras with factorization}\} \\
\cap \\ \{\mbox{Logmodular algebras}\} \\ \cap
\\ \{\mbox{Logrigged algebras} \}.\end{array} \]  These classes include:
certain nest algebras \cite{Pitts:1988} and the finite maximal
subdiagonal algebras of a von Neumann algebra \cite{Ji-Saito:1998}.

We now show that the unique extension property for an operator
algebra is in fact not weaker than the property used in
\cite{Blecher-Labuschagne:2003}.

\begin{prop} Let $A$ be a unital operator algebra then the following are
equivalent: \begin{enumerate} \item[(a)] $A$ has the unique
extension property. \item[(b)] Every $*$-representation of
$C^*_e(A)$ is a unique extension.
\end{enumerate} \end{prop}

\begin{proof} That $b$ implies $a$ is trivial.  To see the other
direction let $\pi: C^*_e(A) \rightarrow B(\mathcal{H})$ be a
completely contractive representation.  Letting $\tau$ be a faithful
representation $\tau: C^*_e(A) \rightarrow B(\mathcal{K})$ then
notice that $ \tau \circ \pi : C^*_e(A) \rightarrow B( \mathcal{K}
\circ \mathcal{H})$ is faithful and hence is a unique extension.  It
now follows that $ \pi$ is a unique extension. \end{proof}

Another class of algebras with the unique extension property are the
Dirichlet algebras.  Recall that  a Dirichlet algebra is an operator
algebra $A$ such that $ A+A^*$ is dense in $C^*_e(A)$, see
\cite{Blecher-Labuschagne:2003}. In this case the unique extension
property is obvious since for any unital completely positive map:
$\pi: A \rightarrow B(\mathcal{H})$ we have that $\pi|_A$ uniquely
defines $\pi|_{A^*}$ and hence it uniquely defines the extension $
\tilde{\pi}: C^*_e(A) \rightarrow B(\mathcal{H})$.

\section{Semicrossed products for multivariable dynamics}

The semicrossed product algebras for multivariable dynamics were
defined in \cite{Davidson-Katsoulis:2007} as a generalization of the
semicrossed products of \cite{Peters:1984}.  Given a locally compact
Hausdorff space $X$ and $\tau_i$ a collection of continuous maps
from $X$ to $X$ there is a universal nonselfadjoint operator algebra
generated by $C_0(X)$ and contractions $S_i$ such that $ f(x)S_i =
S_if(\tau_i(x))$ for all $i$.  We denote this universal algebra by $
C_0(X) \times_{\tau} \mathbb{F}_n^+$, where $\mathbb{F}_n^+$
represents the free semigroup generated by n copies of
$\mathbb{Z}^+$ amalgamated over the identity.

In the case of a single continuous map this is the semicrossed
product defined by Peters.  We first show that these algebras can be
written as universal free products.

\begin{thm} Let $ \tau = (\tau_1,\tau_2, \cdots, \tau_n)$ denote a
collection of continuous self maps of $X$, a locally compact
Hausdorff space.  Then $C_0(X) \times_{\tau} \mathbb{F}_n^+$ is
completely isometrically isomorphic to $*_{C_0(X)} (C_0(X)
\times_{\tau_i} \mathbb{Z}^+)$.\end{thm}

\begin{proof} This is an application of universal properties.
Notice that the algebra $C_0(X) \times_{\tau_i} \mathbb{Z}^+$ is
generated by $C_0(X)$ and a contraction $S_i$ satisfying $ f(x)S_i =
S_if(\tau_i(x))$ for each $i$.  It follows that $ *_{C_0(X)}(C_0(X)
\times_{\tau_i} \mathbb{Z}^+)$ is generated by $C_0(X)$ and
contractions $S_i$ satisfying the covariance conditions.  By
universality there exists a completely contractive homomorphism $
\pi: C_0(X) \times_{\tau} \mathbb{F}_n^+ \rightarrow
*_{C_0(X)} (C_0(X) \times_{\tau_i}\mathbb{Z}^+)$ which is onto.

Similarly since $C_0(X) \times_{\tau} \mathbb{F}_n^+$ is generated
by $C_0(X)$ and contractive operators $S_i$ satisfying the
covariance conditions, there is for each $i$ a completely
contractive homomorphism $\pi_i: C_0(X) \times_{\tau_i} \mathbb{Z}^+
\rightarrow C_0(X) \times_{\tau} \mathbb{F}_n^+$.  Using the
universal property of free products it follows that there is a
completely contractive representation $ * \pi_i: *_{C_0(X)} (C_0(X)
\times_{\tau_i} \mathbb{Z}^+ \rightarrow C_0(X) \times_{\tau}
\mathbb{F}_n^+$.  Which is onto a generating set for $C_0(X)
\times_{\tau} \mathbb{F}_n^+$.

It follows by keeping track of $S_i$ and $C_0(X)$ under the maps
$\pi$ and $ *\pi_i$ that the two algebras are completely
isometrically isomorphic.\end{proof}

Unfortunately it is not immediate that that given a semicrossed
product $ C_0(X) \times_{\alpha} \mathbb{Z}^+$ the algebra has the
unique extension property.  We will show in the next section that
this fact is indeed true.

We focus first on the simple case where $ \alpha$ is surjective. It
is well known under these circumstances $C_0(X) \times_{\alpha}
\mathbb{Z}^+$ can be imbedded completely isometrically
isomorphically into a crossed product algebra $C_0(Y)
\times_{\alpha'} \mathbb{Z}$. In particular the isometry $S$ lifts
to a unitary and hence every element in $C^*_e(C_0(X)
\times_{\alpha} \mathbb{Z}^+)$ can be written as the norm limit of
finite polynomials given by linear combinations of elements of the
form $ U^n f_{n,m} (U^m)^*$ where $U$ is a unitary.

\begin{prop} If $\alpha$ is surjective then the operator algebra
$C_0(X) \times_{\alpha} \mathbb{Z}^+$ is Dirichlet. \end{prop}

\begin{proof} Looking at $x =  U^n f_{n,m} (U^m)^*$, since $U$ is a
unitary we have two cases.  If $n \geq m$ we have $ x = f_{n,m}
\circ \alpha^n U^{n-m} \in C_0(X) \times_{\alpha} \mathbb{Z}^+$ or
if $ n < m$ we get $ x = (U^*)^{m-n} f_{n,m} \circ \alpha^{m} \in
(C_0(X) \times_{\alpha} \mathbb{Z}^+)^*$. The result now follows.
\end{proof}

\section{Semicrossed products with the unique extension property}

For the case that $X$ is metrizable but the $\alpha_i$ are not
necessarily onto we will need a result of \cite{Muhly-Solel:1998a}
and a characterization of when this result applies to semicrossed
products. To do this we remind the reader of the following
definition.

Given Hilbert module $\mathcal{K}, \mathcal{M},$ and $ \mathcal{Q}$
over an operator algebra $A$ we say $K$ is {\em orthogonally
injective} if every contractive short exact sequence of the form \[
0 \rightarrow \mathcal{K} \rightarrow \mathcal{M} \rightarrow
\mathcal{Q} \rightarrow 0 \] has a contractive splitting.  We say
$\mathcal{Q}$ is {\em orthogonally projective} if any contractive
short exact sequence as above, has a contractive splitting. Muhly
and Solel showed in \cite{Muhly-Solel:1998a} that a representation
$\pi: A \rightarrow B(\mathcal{H})$ has the unique extension
property if and only if $\mathcal{H}$ is both orthogonally injective
and orthogonally projective.

For the semicrossed products of \cite{Davidson-Katsoulis:2007} there
is a relatively simple characterization of when a representation
$\pi:A \rightarrow B(\mathcal{H})$ is orthogonally projective and
orthogonally injective. Given a representation $\pi: C_0(X)
\times_{\alpha}\mathbb{Z}^+ \rightarrow B(\mathcal{H})$ there is an
induced representation $\overline{\pi}$ on the algebra of Borel sets
on $X$, denoted $\mathcal{B}(X)$.  The image of $\chi_{\alpha(X)}$
is a spectral projection denoted $E(\alpha(X))$.  The representation
is said to be full if $ \pi(S) \pi(S)^* = E(\alpha(X))$.  Now we
find in Proposition 6.6 of \cite{Davidson-Katsoulis:2007} that if
$\pi: C_0(X) \times_{\alpha}\mathbb{F}_n^+ \rightarrow
B(\mathcal{H})$ is a completely contractive representation then $H$
is orthogonally projective and orthogonally injective if and only if
$\pi$ is a full isometric representation.  To apply Theorem
\ref{main} we will need to verify that a faithful representation of
$C^*_e(C_0(X) \times_{\alpha} \mathbb{Z}^+)$ is a full isometric
representation for $C_0(X) \times_{\alpha} \mathbb{Z}^+$.

\begin{thm} Let $X$ be metrizable and let
$\pi: C^*_e(C_0(X) \times_{\alpha} \mathbb{Z}^+) \rightarrow
B(\mathcal{H})$ be a faithful $*$-representation.  Then
$\pi|_{C_0(X) \times_{\alpha} \mathbb{Z}^+}$ is a full isometric
representation and hence $C_0(X) \times_{\alpha}\mathbb{Z}^+$ has
the unique extension property. \end{thm}

\begin{proof} That $\pi$ is isometric is trivial.  We now apply Theorem 6.5
of \cite{Davidson-Katsoulis:2007} to see that $ \pi|_{C_0(X)
\times_{\alpha} \mathbb{Z}^+}$ has a dilation $ \tilde{\pi}: C_0(X)
\times_{\alpha} \mathbb{Z}+ \rightarrow B(\mathcal{K})$ such that
$\tilde{\pi}$ is a full isometric representation.  In particular if
we denote by $\tilde{E}(\alpha(X))$ the projection in
$B(\mathcal{K})$ corresponding to the Borel function
$\chi_{\alpha(X)} \in \mathcal{B}(X)$, then we know that $
\tilde{\pi}(S) \tilde{\pi}(S^*) = \tilde{E}(\alpha(X))$.  Now since
$\tilde{\pi}$ is isometric we know that there is a contractive
$*$-representation $ \theta: C^*(\tilde{\pi}(C_0(X)
\times_{\alpha}\mathbb{Z}^+)) \rightarrow
\pi(C^*_e(C_0(X)\times_{\alpha}\mathbb{Z}^+))$.  This representation
will induce a $*$-representation $\overline{\theta}$ of
$\overline{\tilde{\pi}}(\mathcal{B}(X))$ onto
$\overline{\pi}(\mathcal{B}(X))$.  It will follow by uniqueness of
spectral measures that $ \overline{\theta}(\tilde{E}(\alpha(X))) =
E(\alpha(X))$.  Now $ \overline{\theta}(\tilde{E}(\alpha(X))) =
\theta(\tilde{\pi}(S) \tilde{\pi}(S)^*)$ and hence $ E(\alpha(X)) =
\pi(S)\pi(S)^*$.  It follows that $ \pi$ is a full isometric
representation.
\end{proof}

We now have the simple corollary that applies to all the semicrossed
products for multivariable dynamics where $X$ is metrizable.

\begin{cor} Let $X$ be a metrizable topological space and $
\alpha_i$ a collection of continuous self maps.  Then $ C^*_e(C(X)
\times_{\alpha} \mathbb{F}_n^+) \cong *_{C(X)} C^*_e(C(X)
\times_{\alpha_i} \mathbb{Z}^+)$. \end{cor}

\bibliographystyle{plain}

\end{document}